\documentclass{amsart}
\newtheorem{theorem}{Theorem}[section]
\newtheorem{lemma}[theorem]{Lemma}

\newtheorem{proposition}[theorem]{Proposition}

\newtheorem{example}[theorem]{Example}
\newtheorem{remark}[theorem]{Remark}
\def\erre{{\rm I\!R}}
\def\R{{\rm I\!R}}

\def\meas{\mathop{\rm meas}}

\def\R{{\rm I\!R}}
\def\R{{\rm I\!R}}

\def\phi{\varphi}
\def\div{\mathop{\rm div}}
\def\meas{\mathop{\rm meas}}
\def\H{\mathop{W^{1,p}(\Omega)}}

\title[Nonlinear Neumann problems...]{Nonlinear Neumann problems driven by a nonhomogeneous differential operator}
\author{Giovanni Molica Bisci}
\address[G. Molica Bisci]{Dipartimento MECMAT, University of Reggio Calabria,
 Via Graziella, Feo di Vito, 89124 Reggio Calabria, Italy.} \email{gmolica@unirc.it}
\author{Du\v{s}an Repov\v{s}}
\address[D. Repov\v{s}]{Faculty of Education, and Faculty of Education Mathematics and Physics, University of Ljubljana, POB 2964, Ljubljana, Slovenia 1001.
}
\email{dusan.repovs@guest.arnes.si}
\thanks{Typeset by \LaTeX}
\begin{document}

\begin{abstract}
We study a nonlinear parametric Neumann problem driven by a nonhomogeneous
quasilinear elliptic differential operator $\operatorname{div}(a(x,\nabla u))$, a
special case of which is the $p$-Laplacian. The reaction term is a nonlinearity
function $f$ which exhibits $(p-1)$-subcritical growth. By using variational methods,
we prove a multiplicity result on the existence of weak solutions for such problems. An explicit example of an application
is also presented.
\end{abstract}
\maketitle
\section{Introduction}
In this paper we study the existence of multiple solutions for the following Neumann problem,
\begin{equation} \tag{$N_{\lambda,\mu}$} \label{Nostro}
\left\{
\begin{array}{ll}
-\operatorname{div}(a(x,\nabla u))+|u|^{p-2}u= (\lambda k(x)+\mu)f(u) & \rm in \quad \Omega
\\ \displaystyle\frac{\partial u}{\partial n_a}=0  & \rm on \,
\partial \Omega.\\
\end{array}
\right.
\end{equation}

Here and in the sequel, $\Omega$ is a bounded, connected domain in $({\R}^{N},|\cdot|)$ with smooth
boundary $\partial \Omega$, $p>1$, $a:\bar\Omega\times {\R}^{N}\to {\R}^N$ is a suitable Carath\'{e}odory map which is strictly monotone in the $\xi\in \erre^N$ variable and $\partial u/\partial n_a:=a(x,\nabla u)\cdot n$, where $n$ is the outward unit normal vector on $\partial \Omega$. Further, $\lambda$ and $\mu$ are positive real parameters, $k\in L^{\infty}(\Omega)_+$ and finally, $f:\erre\rightarrow \erre$ is a continuous function which is $(p-1)$-sublinear at infinity. We cite a recent monograph by Krist\'aly, R\u adulescu and Varga \cite{k2} as a general reference on variational methods.\par

\indent Recently, problems involving $p$-Laplacian-like operators have been studied by several authors under different boundary conditions and by using different technical approaches. \par
For instance, Dirichlet
problems involving a general operator in divergence form were studied by De N\'apoli and Mariani in \cite{n1} by imposing symmetry condition on the map $\xi\mapsto a(a,\xi)$. In the cited paper the existence of one weak solution was proved by exploiting the standard
mountain pass geometry and requiring, among other assumptions, that the nonlinearity $f$ has a $(p-1)$-superlinear behaviour at infinity. The non-uniform case was successively considered by Duc and
Vu in \cite{d1} who extended the result of \cite{n1} under the key hypothesis that the map $a$ fulfills a suitable growth condition.\par
 In \cite{k1}, by using variational methods, Krist\'aly,
Lisei and Varga studied the analogue of the above case for a uniform Dirichlet problem with parameter,
obtaining the existence of three weak solutions requiring that the nonlinearity $f$ has a $(p-1)$-sublinear growth at
infinity.\par
 Successively, Yang, Geng and Yan \cite{y1} proved the existence of three weak solutions for singular
$p$-Laplacian type equations. Finally, Papageorgiou, Rocha and Staicu in
\cite{p1} considered a nonsmooth $p$-Laplacian problem in divergence form, obtaining
the existence of at least two nontrivial weak solutions. See also the contributions obtained by Servadei in \cite{Raffy} for related multiplicity results.\par

The study of the corresponding Neumann problem is in some sense lagging
behind. Superlinear Neumann problems were studied by Aizicovici, Papageorgiou and Staicu
\cite{APS} and Gasi\'{n}ski-Papageorgiou \cite{papa}. In \cite{APS} the differential operator is the
$p$-Laplacian and the superlinear reaction term satisfies the celebrated (AR)-condition. In \cite{papa}
the differential operator is nonhomogeneous incorporating the $p$-Laplacian, but for
the superlinear case the authors prove only an existence theorem and do not have
multiplicity results. Related to this paper are also the nice works \cite{HP,MP2} and references therein.\par

Our goal in this paper is to prove a multiplicity result for Neumann problem $(N_{\lambda,\mu})$ by using a critical point result due to Ricceri (see Theorem \ref{Ric-masodik}). More precisely, for a suitable $\mu=\mu_0$ and $\lambda$ sufficiently small, the existence of multiple solutions for problem $(N_{\lambda,\mu_0})$ will be obtained requiring that the nonlinearity $f$ has a $(p-1)$-linear growth in addition to a suitable oscillating behaviour of the associated potential (see condition $(\textrm{h}_m^{\mu_0})$). We also emphasize that our hypotheses on $a$, following the approach given in \cite{papa}, are considerably weaker than the corresponding ones in \cite{n1,k1}, where $a(x,\xi)=:\nabla_\xi A(x,\xi)$, with $A\in C(\bar \Omega\times \erre^N)$ and for every $x\in \bar \Omega$, $A(x,\cdot)\in C^1(\erre^N)$. Moreover, they assume that for every $x\in\bar \Omega$, the function $\xi\mapsto A(x,\xi)$ is a strongly convex function.\par
 This requirement, in the special case of the $p$-Laplacian operator $\div(|\nabla u|^{p-2}\nabla u)$ implies that $p\geq 2$. In contrast, in our approach we only have that for every $x\in\bar \Omega$, the map $\xi\mapsto A(x,\xi)$ is strictly convex. So, for $p$-Laplacian equations we allow any $p>1$.\par

The plan of the paper is as follows. Section 2 is devoted to our abstract framework, while Section 3 is dedicated to the main results. A concrete example of an application is then presented (see Example \ref{esempio}).
\section{Abstract Framework}
 Let $W^{1,p}(\Omega)$ $(p>1)$ be the
usual Sobolev space, equipped with the norm
$$
\|u\|:=\Big(\int_{\Omega}(|\nabla
u(x)|^{p}+|u(x)|^{p})dx\Big)^{1/p}.
$$

\noindent Further, let $W^{-1,p'}(\Omega)$, with $1/p+1/p'=1$, be its topological
dual and denote the duality brackets for the pair $(W^{-1,p'}(\Omega),W^{1,p}(\Omega))$ by
$\langle\cdot,\cdot\rangle$.
Indicate by $p^*$ the critical exponent of the Sobolev embedding $\H\hookrightarrow L^q(\Omega)$.\par
 Recall that if $p<N$ then $p^*=Np/(N-p)$ and for every $q\in [1,p^*]$ there exists a positive constant $c_q$ such that
\begin{equation}\label{Poi}
\|u\|_{L^q(\Omega)}\leq c_q \|u\|\,,
\end{equation}
\noindent for every $u\in \H$. Moreover, when $p\geq N$, this inequality holds for any $q\in [1,+\infty[$, since $p^*=+\infty$.\par

Our main tool will be the following abstract critical point theorem due to Ricceri \cite{Ri}.

\begin{theorem}\label{Ric-masodik} Let $H$ be a separable and reflexive real Banach space and let
$\mathcal N,\mathcal G:H\to \erre $ be sequentially weakly lower
semicontinuous and continuously G\^ateaux differentiable
functionals, with $\mathcal N$ coercive.\par Assume that the
functional $J_\lambda:=\mathcal N+\lambda\mathcal G$ satisfies the Palais-Smale
condition for every $\lambda>0$ small enough and that the set of all
global minima of $\mathcal N$ has at least $m$ connected components
in the weak topology, with $m\geq 2.$

Then for every $\eta>\displaystyle\inf_{u\in H}\mathcal N(u),$ there exists $\bar
\lambda>0$ such that for every $\lambda\in (0,\bar \lambda),$
the functional $J_\lambda$ has at least $m+1$
critical points, $m$ of which are lying in the set $\mathcal
N^{-1}((-\infty,\eta)).$
\end{theorem}

For the sake of completeness, we also recall that a $C^1$-functional $J:X\to\R$, where $X$ is a real Banach
space with topological dual $X^*$, satisfies the Palais-Smale condition at level $\mu\in\R$, (briefly $\textrm{(PS)}_{\mu}$) when
\begin{itemize}
\item[$\textrm{(PS)}_{\mu}$] {\it Every sequence $\{u_n\}$ in $X$ such that
$$
J(u_n)\to \mu, \qquad{\rm and}\qquad \|J'(u_n)\|_{X^*}\to0,
$$
possesses a convergent subsequence.}
\end{itemize}
Finally, we say that $J$ satisfies the Palais-Smale condition (in short $\textrm{(PS)}$) if $\textrm{(PS)}_{\mu}$ holds for every $\mu\in \erre$.
\section{Main result}

 In the sequel, let $\Omega\subset \erre^N$ be a bounded and connected Euclidean domain. Assume that there exists a function $A:\bar\Omega\times \erre^{N}\to \erre$, with gradient $a(x,\xi):=\nabla_\xi A(x,\xi):\bar \Omega\times \erre^N\rightarrow \erre^{N}$, such that the following conditions hold:
\begin{itemize}
\item[$(\alpha_1)$]\textit{For all} $\xi\in \erre^N$, \textit{the function} $x\mapsto A(x,\xi)$ \textit{is measurable};

\item[$(\alpha_2)$] \textit{For almost all} $x\in \bar \Omega$,\textit{ the function} $\xi\mapsto A(x,\xi)$ is $C^1$, \textit{strictly convex, and} $A(x,0)=0$;

\item[$(\alpha_3)$] \textit{For almost all} $x\in\bar\Omega$ and all $\xi\in \erre^N$, \textit{we assume}
$$
|a(x,\xi)|\leq a_0(x)+c_0|\xi|^{p-1},
$$
\textit{with} $a_0\in L^{\infty}(\Omega)_+$, $c_0>0$ and $p>1$;

\item[$(\alpha_4)$] \textit{For almost all} $x\in\bar \Omega$ and all $\xi\in \erre^N$, \textit{we suppose}
$$
a(x,\xi)\cdot \xi\leq p A(x,\xi);
$$
\item[$(\alpha_5)$] \textit{There exists} $\kappa>0$ \textit{such that for almost all} $x\in\bar\Omega$ \textit{and every} $\xi\in \erre^N$, \textit{we have } $\kappa|\xi|^p\leq p A(x,\xi)$.
\end{itemize}

\begin{example}\rm{We present some examples of functions $A(x,\xi)$ which correspond to the map $a(x,\xi)$ and satisfy the above hypotheses.
\begin{itemize}
\item[$\circ$] $A(x,\xi):=\displaystyle\frac{|\xi|^p}{p}$ with $p>1$. Then
$$
a(x,\xi):=\nabla_\xi A(x,\xi)=|\xi|^{p-2}\xi.
$$
In this setting, the resulting differential operator is the usual $p$-Laplacian
$$
\Delta_p u:=\div(|\nabla u|^{p-2}\nabla u);
$$

\item[$\circ$] $A(x,\xi):=\displaystyle\frac{a_1(x)}{p}|\xi|^p+\frac{a_2(x)}{p}|\xi|^r$, with $a_1,a_2\in L^\infty(\Omega)_+$, $a_1(x)\geq c_0>0$ for almost every $x\in\bar\Omega$ and $1<r<p$;

    \item[$\circ$] $A(x,\xi):=\displaystyle\frac{a_1(x)}{p}|\xi|^p+\frac{1}{r}\log(1+|\xi|^r)$, with $a_1\in L^\infty(\Omega)_+$, $a_1(x)\geq c_0>0$ for almost every $x\in\bar\Omega$ and $1<r\leq p$;
        \item[$\circ$] $A(x,\xi):=\displaystyle\frac{1}{p}((1+|\xi|^2)^{p/2})$, with $p>1$. Thus
        $$
        a(x,\xi)=(1+|\xi|^2)^{(p-2)/2}\xi.
        $$
        The resulting differential operator is the generalized mean curvature operator
        $$
        \div((1+|\nabla u|^2)^{(p-2)/2}\nabla u);
        $$

        \item[$\circ$] $A(x,\xi):=\displaystyle \frac{M(x)\xi\cdot \xi}{2}$, with $M\in L^\infty(\bar\Omega;\erre^{N\times N})$ and $M(x)\geq c_0 I_N$ for almost every $x\in \bar\Omega$, with $c_0>0$ and $I_N$ being the identity $N$-matrix.
\end{itemize}
}
\end{example}

\begin{remark}\label{S+}\rm
 The
operator $a(x,\xi):=\nabla_{\xi}A(x,\xi)$ satisfies the $(S_+)$ property; see \cite[Proposition 3.1]{papa}. This means that for every sequence $\{u_n\}\subset W^{1,p}(\Omega)$ such that $u_n\rightharpoonup u$ (weakly) in $W^{1,p}(\Omega)$ and
$$
\limsup_{n\to\infty}\int_{\Omega}a(x,\nabla u_n(x))
 \cdot\nabla(u_n-u)(x)dx\leq0,$$
 then $u_n\rightarrow u$ (strongly) in $W^{1,p}(\Omega)$.
\end{remark}

From now on, let $f:\erre\rightarrow \erre$ be a continuous function such that
\begin{itemize}
\item [$(\textrm{h}_\infty)$]
$
\displaystyle\lim_{|t|\rightarrow \infty}\frac{f(t)}{|t|^{p-1}}=0.
$
\end{itemize}
A typical case when $(\textrm{h}_\infty)$ holds is
\begin{enumerate}
\item[$({\rm h}_{q\rho})$] {\it There exist $q\in (0,p-1)$ and $\rho>0$ such that $|f(t)|\leq
\rho|t|^{q}$ for every $t\in \erre .$}
\end{enumerate}
\indent In order to obtain our multiplicity result, in addition to condition $(\textrm{h}_\infty)$, we also require that:
\begin{enumerate}
\item[$(\textrm{h}_m^{\mu_0})$] {\it There exists $\mu_0\in (0,\infty)$ such that the set of
global minima of the function $$s\mapsto \tilde F_{\mu_0}
(s):=\Lambda{s^p}-\mu_0 F(s),$$ has at least $m\geq 2$ connected
components.}
\end{enumerate}

\noindent Note that $(\textrm{h}_m^{\mu_0})$ implies that the function $s\mapsto
\tilde F_{\mu_0} (s)$ has at least $m-1$ local maxima.\par

We are interested in the existence of multiple weak solutions for the following
Neumann problem
\begin{equation} \tag{$N_{\lambda,\mu_0}$} \label{Nostro}
\left\{
\begin{array}{ll}
-\operatorname{div}(a(x,\nabla u))+|u|^{p-2}u= (\lambda k(x)+\mu_0)f(u) & \rm in \quad \Omega
\\ \displaystyle\frac{\partial u}{\partial n_a}=0  & \rm on \,
\partial \Omega.\\
\end{array}
\right.
\end{equation}

For the sake of completeness we recall that, fixing $\lambda>0$, a \textit{weak solution} of problem \eqref{Nostro} is a function $u\in W^{1,p}(\Omega)$ such that
\begin{eqnarray*}
\int_{\Omega}a(x,\nabla u(x))\cdot\nabla
v(x)\,dx &=& -\int_{\Omega}|u(x)|^{p-2}u(x)v(x)\,dx\\\nonumber
&+&\lambda\int_\Omega k(x)f(u(x)) v(x) \,dx\\\nonumber
&+&\mu_0\int_\Omega f(u(x))v(x)\,dx,
\end{eqnarray*}
\noindent for every $v\in W^{1,p}(\Omega)$.\par
\indent Set $\Phi: W^{1,p}(\Omega)\rightarrow \erre$ given by
\[
\Phi(u):=\int_{\Omega}A(x,\nabla
u(x))dx+\frac{1}{p}\int_{\Omega}|u(x)|^pdx,
\]
and
$$
\mathcal  N_{\mu_0}(u):=\Phi(u)-{\mu_0}\int_\Omega F(u(x))dx,
$$
as well as
$$
\mathcal{G}(u):=-\int_{\Omega}k(x)F(u(x))dx,$$
for every $u\in W^{1,p}(\Omega)$.
Here, as usual, we put
$$
F(s):=\int_0^{s}f(t)dt,
$$
for every $s\in\erre$.\par

With the above notations and assumptions, it is easy to prove that $\mathcal{N}_{\mu_0}$ and $\mathcal{G}$ are $C^1$-functionals with
G\^{a}teaux derivatives given by
\begin{eqnarray*}
\langle\mathcal  N_{\mu_0}'(u),v\rangle=\int_{\Omega}a(x,\nabla u(x))\cdot\nabla
v(x)dx &+& \int_{\Omega}|u(x)|^{p-2}u(x)v(x)dx\\\nonumber
&-&\mu_0\int_\Omega f(u(x))v(x)dx,
\end{eqnarray*}
and
$$
\langle\mathcal{G}'(u),v\rangle=-\int_{\Omega}k(x)f(u(x))v(x)dx,$$

\noindent for every $v\in W^{1,p}(\Omega)$.\par
Thus, the critical points of $J_\lambda:=\mathcal{N}_{\mu_0}+\lambda\mathcal{G}$ are exactly the weak solutions of problem $(P_\lambda)$.\par

Finally, denote
$$
\Lambda:=\frac{\min\left\{\kappa,1\right\}}{p}.
$$

\indent Standard arguments ensure the validity of the following preliminary regularity result on the functionals $\mathcal{N}_{\mu_0}$ and $\mathcal{G}$.

\begin{lemma}\label{lemma1}
 Let us assume that condition $({\rm h}_\infty)$ holds. Then the above functionals $\mathcal{N}_{\mu_0}$ and $\mathcal{G}$ are
sequentially weakly lower semicontinuous.
\end{lemma}

\begin{proof}
 Due to condition $(\alpha_2)$ the functional $\Phi$ is convex. Since $\Phi$ is strongly continuous it is also
weakly lower semicontinuous. On the other hand, since condition $({\rm h}_\infty)$ holds, there exists a positive constant $c$ such that
$|f(t)|\leq c(1+|t|^{p-1})$, for every $t\in\erre$. Finally, due to the fact that
the embedding $X \hookrightarrow
L^{p}(\Omega)$ is compact, we obtain that the functionals
$$
u\mapsto -\int_\Omega F(u(x))dx,\quad{\rm and}\quad u\mapsto -\int_{\Omega}k(x)F(u(x))dx,
$$
\noindent are sequentially weakly lower semicontinuous by arguing in standard way.
\end{proof}

Further, the $C^1$-functional
$J_\lambda$ satisfies the $\textrm{(PS)}$-condition as proved in the next result.
\begin{lemma}\label{lemma2}
Assume that condition $(\rm{h}_\infty)$ holds. Then the functional $J_\lambda$ is coercive and
satisfies the $\rm{(PS)}$-condition for every real parameter $\lambda$.
\end{lemma}

\begin{proof}
\noindent Let us fix $\lambda\in \erre$ and consider
$$
0<\alpha<\frac{1}{\mu_0+|\lambda|\|k\|_\infty}.
$$
By condition $(\textrm{h}_\infty)$, there exists $\delta_\lambda$ such that
$$
|f(t)|\leq \frac{\alpha p\Lambda}{c_p^p}|t|^{p-1},
$$
\noindent for every $|t|\geq \delta_\lambda$. By integration we have
$$
|F(s)|\leq \frac{\alpha\Lambda}{c_p^p}|s|^{p}+\max_{|t|\leq \delta_\lambda}|f(t)||s|,
$$
\noindent for every $s\in\erre$.\par
\indent Thus, by using the above inequality and bearing in mind relation \eqref{Poi}, one has
\begin{eqnarray}\label{e3.2a}
J_\lambda(u) &\geq& \Phi(u)-{\mu_0}\left|\int_\Omega F(u(x))dx\right|-|\lambda| |\mathcal{G}(u)|\nonumber\\
               &\geq& \Lambda(1-\alpha(\mu_0+|\lambda|\|k\|_\infty))\|u\|^p\nonumber\\
               &-&c_1(\mu_0+|\lambda|\|k\|_\infty )\max_{|t|\leq \delta_\lambda}|f(t)|\|u\|,\nonumber
\end{eqnarray}
\noindent where $p':=(p-1)/p$ is, as usual, the conjugate exponent of $p$.
\noindent Then the functional $J_\lambda$ is bounded from below and, since $p>1$,
$J_\lambda(u)\rightarrow +\infty$ whenever $\|u\|\rightarrow +\infty$. Hence $J_\lambda$ is coercive.\par
 Now, fix $\mu\in\erre$ and let us prove that $J_\lambda$ satisfy the condition $\textrm{(PS)}_{\mu}$. For this goal, let $\{u_{n}\}\subset W^{1,p}(\Omega)$ be a Palais-Smale
sequence, i.e.
$$
J_\lambda(u_n)\to \mu, \qquad{\rm and}\qquad \|J'_\lambda(u_n)\|_{W^{-1,p'}}\to0.
$$
Taking into account the coercivity of $J_\lambda$, the sequence $\{u_n\}$ is necessarily bounded in $W^{1,p}(\Omega)$. Since $W^{1,p}(\Omega)$ is reflexive, we may extract a subsequence
that for simplicity we call again $\{u_{n}\}$, such that
$u_{n}\rightharpoonup u$ in $W^{1,p}(\Omega)$.\par
 We will prove that ${u_{n}}$ strongly converges to $u\in W^{1,p}(\Omega)$.
Exploiting the derivative $J_\lambda(u_n)(u_n-u)$, we obtain
\begin{eqnarray*}
\int_{\Omega}a(x,\nabla
u_n(x))\cdot\nabla(u_n-u)(x)dx &=& \langle J'_\lambda(u_n),u_n-u\rangle\\\nonumber
&-&\int_{\Omega}|u_n(x)|^{p-2}u_n(x)(u_n-u)(x)dx \nonumber\\
&- & \mu_0\int_{\Omega}f(u_n(x))(u_n-u)(x)dx\nonumber\\
               &- & \lambda\int_{\Omega}k(x)f(u_n(x))(u_n-u)(x)dx.\nonumber
\end{eqnarray*}

Since $\|J'_\lambda(u_n)\|_{W^{-1,p}}\to0$ and the sequence $\{u_n-u\}$ is bounded in
$W^{1,p}(\Omega)$, taking into account that $|\langle
J'_\lambda(u_n),u_n-u\rangle|\leq\|J'_\lambda(u_n)\|_{W^{-1,p'}}\|u_n-u\|$, one has
\[
\langle J'_\lambda(u_n),u_n-u\rangle\to0.
\]
 \indent Further, by the asymptotic condition $(\textrm{h}_\infty)$, there exists a real positive constant $c$ such that $|f(t)|\leq c(1+|t|^{p-1})$, for every $t\in\erre$. Then
\begin{align*}
&\int_{\Omega}|f(u_n(x))||u_n(x)-u(x)|dx \\
&\leq c\left(\int_{\Omega}|u_n(x)-u(x)|dx
+ \int_{\Omega}|u_n(x)|^{p-1}|u_n(x)-u(x)|dx\right) \\
&\leq c((\meas(\Omega))^{1/p'}+\|u_n\|_{L^{p}}^{p-1}) \|u_n-u\|_{L^p}.
\end{align*}

Now, the embedding
$W^{1,p}(\Omega)\hookrightarrow L^p(\Omega)$ is compact, hence $u_n\to u$ strongly in $L^p(\Omega)$.
So we obtain
\[
\int_{\Omega}|f(u_n(x))||u_n(x)-u(x)|dx\to0.
\]
Analogously, one has
\[
\int_{\Omega}k(x)|f(u_n(x))||u_n(x)-u(x)|dx\to0.
\]
\indent Moreover,
considering the inequality
\begin{align*}
\int_{\Omega}||u_n(x)|^{p-2}u_n(x)(u_n(x)-u(x))|dx
&= \int_{\Omega}|u_n(x)|^{p-1}|u_n(x)-u(x)|dx \\
&\leq \|u_n\|_{L^p}^{p-1}\|u_n-u\|_{L^p},
\end{align*}
and $u_n\to u$ strongly in $L^p(\Omega)$, we have
\[
\int_{\Omega}||u_n(x)|^{p-2}u_n(x)(u_n(x)-u(x))|dx \rightarrow 0.
\]

 \indent We can conclude that
\[ 
\limsup_{n\to\infty}\langle a(x,u_{n}),u_{n}-u\rangle
\leq 0,
\]
where $\langle a(x,u_{n}),u_{n}-u\rangle$ denotes
$$\displaystyle\int_{\Omega}a(x,\nabla u_n(x))\cdot\nabla(u_n-u)(x)dx.$$

\indent But as observed in Remark \ref{S+}, the operator $\Phi'$ has the $(S_+)$ property. So, in conclusion, $u_{n}\to u$ strongly in $W^{1,p}(\Omega)$.\par
 Hence, $J_\lambda$ is bounded from below and fulfills $($\rm{PS}$)$, for every positive parameter $\lambda$.
\end{proof}
\begin{remark}\rm{
We observe that by the above lemma, the functional
$$
 J_0=\mathcal N_{\mu_0}(u):=\Phi(u)-{\mu_0}\int_\Omega F(u(x))dx,\,\,\,\,(u\in W^{1,p}(\Omega))
$$
is coercive.}
\end{remark}
\begin{proposition}\label{lok-minim} The set of all global minima of the functional $\mathcal 
N_{\mu_0}$ has at least $m$ connected components in the weak
topology on $W^{1,p}(\Omega)$.
\end{proposition}
\begin{proof}
First, for every $u\in W^{1,p}(\Omega)$ we
have
\begin{eqnarray*}
\mathcal N_{\mu_0}(u)&= &
\Phi(u)-{\mu_0}\int_\Omega F(u(x))dx\\&\geq &
\Lambda\int_\Omega|\nabla u(x)|^pdx+\int_\Omega\tilde F_{\mu_0}(u(x))dx\\&\geq &
\left(\inf_{s\in \erre }\tilde F_{\mu_0}(s)\right)\meas(\Omega).
\end{eqnarray*}
\indent Moreover, if we consider $u(x)=u_{\tilde
s}(x)=\tilde s$ for almost every $x \in \Omega,$ where $\tilde s\in
\erre $ is a minimum point of the function $s\mapsto \tilde
F_{\mu_0}(s),$ then we have the equality from the previous estimate (note that $\Phi(0)=0$ by using the last part of condition $(\alpha_2)$).
Thus,  $$\inf_{u\in W^{1,p}(\Omega)}\mathcal N_{\mu_0}(u)=
\left(\inf_{s\in \erre }\tilde F_{\mu_0}(s)\right)\meas(\Omega).$$\par
\noindent Further, if
$u\in W^{1,p}(\Omega)$ is not a constant function, we have

\begin{eqnarray*}
\mathcal
N_{\mu_0}(u) &\geq & \Lambda\int_\Omega|\nabla u(x)|^pdx+\int_\Omega\tilde F_{\mu_0}(u(x))dx\nonumber\\ &>&\left(\inf_{s\in \erre }\tilde F_{\mu_0}(s)\right)\meas(\Omega).
\end{eqnarray*}

Consequently, between the sets $${\rm
Min}({\mathcal N_{\mu_0}})=\left\{u\in W^{1,p}(\Omega):\mathcal
N_{\mu_0}(u)=\inf_{u\in W^{1,p}(\Omega)}\mathcal
N_{\mu_0}(u)\right\},$$ and $${\rm Min}(\tilde F_{\mu_0})=\left\{s\in \erre
:\tilde F_{\mu_0}(s)=\inf_{s\in \erre }\tilde F_{\mu_0}(s)\right\},$$ there
is a one-to-one correspondence.\par
\indent Indeed, let $\theta$ be the
function which associates to every number $s\in \erre $ the
equivalence class of those functions which are almost everywhere equal to $s$
in $\Omega$.\par
\indent Then $\theta:{\rm Min}(\tilde
F_{\mu_0})\to {\rm Min}({\mathcal N_{\mu_0}})$ is actually a
homeomorphism between ${\rm Min}(\tilde F_{\mu_0})$ and ${\rm
Min}({\mathcal N_{\mu_0}})$, where the set ${\rm Min}({\mathcal
N_{\mu_0}})$ is considered with the relativization of the weak
topology on $W^{1,p}(\Omega).$\par
 On account of the hypothesis $(\textrm{h}_m^{\mu_0})$, the
set ${\rm Min}(\tilde F_{\mu_0})$ contains at least $m\geq 2$
connected components. Therefore, the same is true for the set
${\rm Min}({\mathcal N_{\mu_0}})$, which completes the proof.
\end{proof}
Our main result is as follows.
\begin{theorem}\label{theorem-15}
Let $f:\erre \to \erre $ be a continuous function such that conditions $({\rm{h}}_\infty)$ and $({\rm{h}}_m^{\mu_0})$ hold. Then
\begin{enumerate}
\item[\textrm{a)}]  For every $\eta>0$, there exists a number $\tilde
\lambda_\eta>0$ such that for every $\lambda\in (0,\tilde
\lambda_\eta)$ problem $(N_\lambda)$ has at least $m+1$ weak solutions
$u_\lambda^{1,\eta},\dots,u_\lambda^{m+1,\eta}\in
W^{1,p}(\Omega);$ \textit{and}
\item[b)] If  $({\rm h}_{q\rho})$ holds  then for each
$\lambda\in (0,\tilde \lambda_\eta)$ there is a set
$I_\lambda\subset \{1,\dots,m+1\}$ with {\rm card}$(I_\lambda)=m$
such that  $$
\|u_\lambda^{i,\eta}\|<t_{\eta q \rho},\
\ \ \ \ (i\in I_\lambda)$$ where $t_{\eta q \rho}>0$ is the greatest
solution of the equation $$\Lambda t^p-\rho\mu_0\frac{
\meas(\Omega)^{((p-1)-q)/p}}{q+1}t^{q+1}-\eta=0,\ \ (t>0).$$
\end{enumerate}
\end{theorem}

\begin{proof}
Let us choose
$H=W^{1,p}(\Omega),$ and
$$
\mathcal{N}:=\mathcal{N}_{\mu_0}=\Phi(u)-{\mu_0}\int_\Omega F(u(x))dx,
$$
as well as
$$
\mathcal{G}(u):=-\int_{\Omega}k(x)F(u(x))dx,$$
for every $u\in W^{1,p}(\Omega)$,
 in Theorem \ref{Ric-masodik}.\par
  Due to Proposition
\ref{lok-minim}, Lemmas \ref{lemma1} and  \ref{lemma2} all the hypotheses of Theorem
\ref{Ric-masodik} are satisfied.\par Note that $\mathcal
N(0)=0$, so $\inf_{u\in H}\mathcal
N(u)\leq 0$. Therefore,
for every
$$
\eta>0\geq\inf_{u\in H}\mathcal
N(u),$$ there is a number $\tilde
\lambda_\eta>0$ such that for every $\lambda\in (0,\tilde
\lambda_\eta)$ the function $\mathcal N_{\mu_0}+\lambda \mathcal
G$ has at least $m+1$ critical points; let us denote them by
$u_\lambda^{1,\eta},\dots,u_\lambda^{m+1,\eta}\in
H.$ Clearly, they are solutions of problem $(N_\lambda),$
which proves the first claim.

We know in addition that $m$ elements from
$u_\lambda^{1,\eta},\dots,u_\lambda^{m+1,\eta}$  belong
to the set $\mathcal N_{\mu_0}^{-1}((-\infty,\eta)).$ Let
$\tilde u$ be such an element, i.e.,
$$
  \mathcal
N_{\mu_0}(\tilde u)=\Phi(\tilde u)-{\mu_0}\int_\Omega F(\tilde u(x))dx<\eta.
$$
Hence, one has
\begin{equation}\label{N-kisebb}
\Lambda\|\tilde u\|^p-{\mu_0}\int_\Omega F(\tilde u(x))dx<\eta.
\end{equation}
 Assume that
 $({\rm h}_{q\rho})$ holds. Then $|F(t)|\leq \displaystyle\frac{\rho}{q+1}|t|^{q+1}$
for every $t\in \erre .$\par By using the H\"older inequality, one has
\begin{equation}\label{estimat-alpha}
  \int_\Omega |\tilde u(x)|^{q+1}dx\leq
\meas(\Omega)^{((p-1)-q)/p}\|\tilde u\|^{q+1}.
\end{equation}

 On account of (\ref{N-kisebb}) and
(\ref{estimat-alpha}) it follows that
\begin{equation}\label{ultima}
\Lambda\|\tilde u\|^p-\rho\mu_0\frac{
\meas(\Omega)^{((p-1)-q)/p}}{q+1}\|\tilde u\|^{q+1}<\eta.
\end{equation}

\indent Now, observe that, since $\eta>0$ and  $q\in (0,p-1)$, it is easy to see that the following algebraic equation
\begin{equation}\label{utolso-egy}
  \Lambda t^p-\rho\mu_0\frac{
\meas(\Omega)^{((p-1)-q)/p}}{q+1}t^{q+1}-\eta=0,
\end{equation}
always has a positive solution.\par
 Finally, bearing in mind \eqref{ultima}, the number $\|\tilde u\|$
is less than the  greatest solution $t_{\eta q \rho}>0$ of the
equation (\ref{utolso-egy}). The proof is complete.
\end{proof}

In conclusion we present a direct and easy application of Theorem \ref{theorem-15} for an elliptic Neumann problem involving the Laplace operator.

\begin{example}\label{esempio}\rm
Let $k\in L^{\infty}(\Omega)_+$ and $f:\erre \to \erre $ be the continuous function defined by $f(t):=\min\{t_+-\sin(\pi
t_+),2(m-1)\}$, where $m\geq 2$ is fixed and $t_+=\max\{t,0\}.$ Consider the following Neumann problem
\begin{equation} \tag{$\widetilde{N}_{\lambda,1}$} \label{Nostro3}
\left\{
\begin{array}{ll}
-\Delta u+u= (\lambda k(x)+1)f(u) & \rm in \quad \Omega
\\ \displaystyle\frac{\partial u}{\partial n}=0  & \rm on \,
\partial \Omega.\\
\end{array}
\right.
\end{equation}
Owing to Theorem \ref{theorem-15}, for every $\eta>0$, there exists a number $\tilde
\lambda_\eta>0$ such that for every $\lambda\in (0,\tilde
\lambda_\eta)$ problem \eqref{Nostro3} has at least $m+1$ weak solutions
$u_\lambda^{1,\eta},\dots,u_\lambda^{m+1,\eta}\in
W^{1,2}(\Omega).$ Indeed, clearly, $(\textrm{h}_\infty)$ holds,
while for $\mu_0=1,$ the assumption $(\textrm{h}_m^{1})$ is also fulfilled.
Indeed, the function $t\mapsto \tilde F_1(t)$ has precisely $m$
global minima; they are $0,2,\dots, 2(m-1).$ Moreover, $\min_{t\in
\erre}\tilde F_{1} (t)=0.$
\end{example}

\begin{remark}\rm{
We emphasize that there are several multiplicity results for nonlinear Neumann problems driven by the $p$-Laplacian differential operator. We mention, among others, the works \cite{BC, BMolica, DM, faraci, MP}. With exception of \cite{DM} and \cite{MP}, in all the cited papers, it is assumed that $p>N$ and the authors exploit the fact that, in this context, the Sobolev space $W^{1,p}(\Omega)$ is compactly embedded in $C^{0}(\bar \Omega)$.}
\end{remark}

\begin{remark}\rm{
For completeness we also cite a recent interesting paper of Colasuonno, Pucci, and Varga \cite{CPV} which contains some multiplicity results on elliptic problems with either Dirichlet or Robin boundary conditions and involving a general operator in divergence form. Moreover, some contributions for nonlinear problems involving a general operator not in divergence form are contained in \cite{MS1, MS2, Raffy2}. Finally, our abstract methods can be also used studying fractional laplacian equations. See, for instance, the manuscript \cite{Raffy3} and references therein for related topics.}
\end{remark}
\medskip
 \indent {\bf Acknowledgements.}
This paper was written when the first author was visiting professor at the University of Ljubljana in 2012. He expresses his gratitude  for the warm hospitality.
 The research was supported in part by the SRA grants P1-0292-0101 and J1-4144-0101.

\end{document}